\documentclass [11pt]{article}
\usepackage{amsthm}
\usepackage{amsmath}
\usepackage{amssymb}
\usepackage{enumerate}
\usepackage{epsfig}
\usepackage{srcltx}
\usepackage[dvipsnames,usenames]{color}

\newcounter{contador}
\newtheorem{propo}[contador]{Proposition}
\newtheorem{teo}[contador]{Theorem}
\newtheorem{lem}[contador]{Lemma}
\newtheorem{nota}[contador]{Remark}

\newtheorem{example}[contador]{Example}

\newcommand{\rec}{\noindent} 
\newcommand{\fiexemple}{\, \hfill$\Box$} 

\newcommand{\dps}{\displaystyle} 

\newcommand{\R}{{\mathbb R}}
\newcommand{\C}{{\mathbb C}}
\newcommand{\N}{{\mathbb N}}
\newcommand{\Z}{{\mathbb Z}}
\newcommand{\Q}{{\mathbb Q}}

\newcommand{\U}{{\cal{U}}}

\title{Non-autonomous 2-periodic \\Gumovski-Mira difference equations \footnote{{\bf Acknowledgements}. GSD-UAB and CoDALab Groups are
supported by the Government of Catalonia through the SGR program. They are also supported by MCYT through grants MTM2008-03437 (first and second
authors) and DPI2008-06699-C02-02 (third author).}}

\author{Anna Cima$^{(1)}$, Armengol Gasull$^{(1)}$ and V\'{\i}ctor Ma\~{n}osa $^{(2)}$
\\*[.1truecm]
{\small \textsl{$^{(1)}$ Dept. de Matem\`{a}tiques, Facultat de Ci\`{e}ncies,}}
\\*[-.25truecm] {\small \textsl{Universitat Aut\`{o}noma de Barcelona,}}
\\*[-.25truecm] {\small \textsl{08193 Bellaterra, Barcelona, Spain}}
\\*[-.25truecm] {\small \textsl{cima@mat.uab.cat, gasull@mat.uab.cat}}
\\*[-.25truecm]
\\*[-.25truecm] {\small \textsl{$^{(2)}$ Dept. de Matem\`{a}tica Aplicada III (MA3),}}
\\*[-.25truecm] {\small \textsl{Control, Dynamics and Applications Group (CoDALab)}}
\\*[-.25truecm] {\small \textsl{Universitat Polit\`{e}cnica de Catalunya (UPC)}}
\\*[-.25truecm] {\small \textsl{Colom 1, 08222 Terrassa, Spain}}
\\*[-.25truecm] {\small \textsl{victor.manosa@upc.edu}}}

\begin{document}
\maketitle

\begin{abstract}
We consider two types of non-autonomous $2$-periodic Gumovski-Mira
difference equations. We show that while the corresponding
autonomous recurrences are conjugated, the behavior of the sequences
generated by the $2$-periodic ones differ dramatically:  in one case
the behavior of the sequences  is simple (\textsl{integrable}) and
in the other case it is much more complicated (\textsl{chaotic}). We
also present a global study  of the integrable case that includes
which periods appear for the recurrence.
\end{abstract}

\rec {\sl Keywords:} Integrable and chaotic difference equations and
maps; rational difference equations with periodic coefficients;
perturbed twist maps. \newline

\rec {\sl 2000 Mathematics Subject Classification:} \texttt{39A20, 39A11}


\section{Introduction}

Autonomous difference equations are a classical tool for the
modeling of ecological systems, see for instance \cite{H,M,T}. One
of the modifications applied in the models in order to adapt them to
more realistic situations consists in converting the recurrences
into non-autonomous ones changing one of the constant parameters by
a periodic cycle (see \cite{C,CH,CH2,ES1,ES2,Sv} and references
therein). In this situation it is said that the model takes into
account seasonality. For instance a parameter taking values in a
cycle of period 4 could model an ecological situation that has
different features during spring, summer, autumn or winter.

In this paper we will consider a very simple autonomous recurrence, which depends on a unique parameter and another one which is conjugated
 to this one. Then we will show that  changing the constant parameter by a two cycle has a completly different effect in both cases.
  Indeed in one of them the recurrence has an invariant while in the second one the recurrence looks as a chaotic one.
  This phenomenon is very surprising and shows that this procedure used in modelling of changing the constant parameters by periodic cycles is quite delicate.

More specifically, we take
\begin{equation}\label{miraa}
\quad x_{n+2}=-x_{n}+\frac{\alpha x_{n+1}}{1+x_{n+1}^2}\quad\mbox{ with }\quad \alpha>0,
\end{equation}
which is one of the recurrences considered by  Gumovski and Mira in
\cite{GM}. Performing the change of variables $x_n=\sqrt{\alpha}\,
y_n,$ it writes as
\begin{equation}\label{mirab}y_{n+2}=-y_{n}+\frac{ y_{n+1}}{\beta+y_{n+1}^2}\quad\mbox{ with }\quad \beta=1/\alpha>0.
\end{equation}

We will see that when considering the corresponding $2$-periodic
recurrences
\begin{equation}\label{miracaotica}
x_{n+2}=-x_{n}+\frac{\alpha_n x_{n+1}}{1+x_{n+1}^2},
\end{equation}
and
\begin{equation}\label{mirainte}
x_{n+2}=-x_{n}+\frac{ x_{n+1}}{\beta_n+x_{n+1}^2},
\end{equation}
where $\{\alpha_n\}_n$ is a $2$-periodic cycle of positive values and $\beta_n=1/\alpha_n,n\ge0$, their behaviors are completely different. For the second case it is very
simple, indeed \textsl{integrable} (that is, it has a two-periodic invariant), see Theorem \ref{main}. On the other hand, for the first
recurrence and suitable initial conditions the points $\{x_n\}_n$ seem to fill densely many disjoint closed intervals. In fact its unfolding $(x_{n},x_{n+1})$ presents all the complexity of the perturbed twist maps, see Section~\ref{se:3} and also Figure 1.

More specifically, in Subsection \ref{recu-g}, we  characterize all the sequences generated by the recurrence \eqref{mirainte}. They are:

\begin{enumerate}[(i)]
\item Constant sequences or periodic sequences with period $2q$ for some $q\in\N$;
\item Sequences $\{x_n\}_{n\in\N}$ such that its adherence is formed by the same orbit plus an accumulation point;
\item Sequences $\{x_n\}_{n\in\N}$ such that its adherence is formed by one closed interval;
\item Sequences $\{x_n\}_{n\in\N}$ such that its adherence is formed by two closed intervals.

\end{enumerate}

With respect to recurrence \eqref{miracaotica}, we will see in Section \ref{se:3} that there appear the same four types of sequences given above plus other ones:

\begin{enumerate}
\item[(v)] Sequences $\{x_n\}_{n\in\N}$ such that its adherence is formed by the same orbit plus some accumulation points;

\item[(vi)] Sequences $\{x_n\}_{n\in\N}$ such that its adherence is formed by several closed intervals.
\end{enumerate}

The differences are perhaps more clear when considering the \textsl{composition maps} (\cite{CGM11a,CGM11b}) associated to the non-autonomous
difference equations, which in this case are given by $F_{\alpha_2,\alpha_1}:=F_{\alpha_2}\circ F_{\alpha_1}$ and $G_{\beta_2,\beta_1}:=G_{\beta_{2}}\circ G_{\beta_1}$, where
$$
F_{\alpha}(x,y)=\left(y,-x+\frac{\alpha y}{1+y^2}\right) \quad\mbox{ and }\quad G_{\beta}(x,y)=\left(y,-x+\frac{y}{\beta+y^2}\right),
$$
are the corresponding maps associated to the recurrences in (\ref{miraa}) and \eqref{mirab}.
Notice that for instance
\[
(x_1,x_2)\xrightarrow{G_{\beta_1}}(x_2,x_3)\xrightarrow{G_{\beta_2}}(x_3,x_4)
\xrightarrow{G_{\beta_1}}(x_4,x_5)\xrightarrow{G_{\beta_2}}(x_5,x_6)\xrightarrow{G_{\beta_1}}\cdots
\]
where $\{x_k\}_{k\in\N}$ is the sequence generated by~\eqref{mirainte} and
similarly for the recurrence~\eqref{miracaotica}.

For all $\alpha$ and $\beta$ the maps $F_\alpha$ and $G_\beta$  are
\textsl{integrable diffeomorphisms} from $\R^2$ to $\R^2$, which
preserve area, and the dynamics on the level sets of their
corresponding  first integrals is \textsl{translation-like}
(\cite{BR2,CGM}). In a few words this means that when the level sets
are circles it is conjugated to a rotation and when they are
diffeomorphic to the real line, it is conjugated to a translation on
this line.

The curious phenomenon that we present is that while the maps $F_{\alpha}$ and $G_{1/\alpha}$ are conjugated via the linear change
$\Psi(x,y)=(\sqrt{\alpha}x,\sqrt{\alpha}y)$ (that is $G_{1/\alpha}=\Psi^{-1}\circ F_{\alpha}\circ \Psi$), and both maps posses a first integral, the dynamics of
the compositions maps are very different: while $F_{\alpha_2,\alpha_1}$ numerically exhibits all the features of a non-integrable perturbed twist map, the map $G_{{1}/{\alpha_2},{1}/{\alpha_1}}$ has a first integral given by
$$
V(x,y)=\frac{1}{\alpha_1} x^2+ \frac{1}{\alpha_2} y^2+x^2y^2-xy,
$$
and its dynamics is \textsl{translation-like} (see Theorem \ref{rotacions}). These two different behaviors are illustrated in Figure 1 below.

\begin{center}
\includegraphics[scale=0.32]{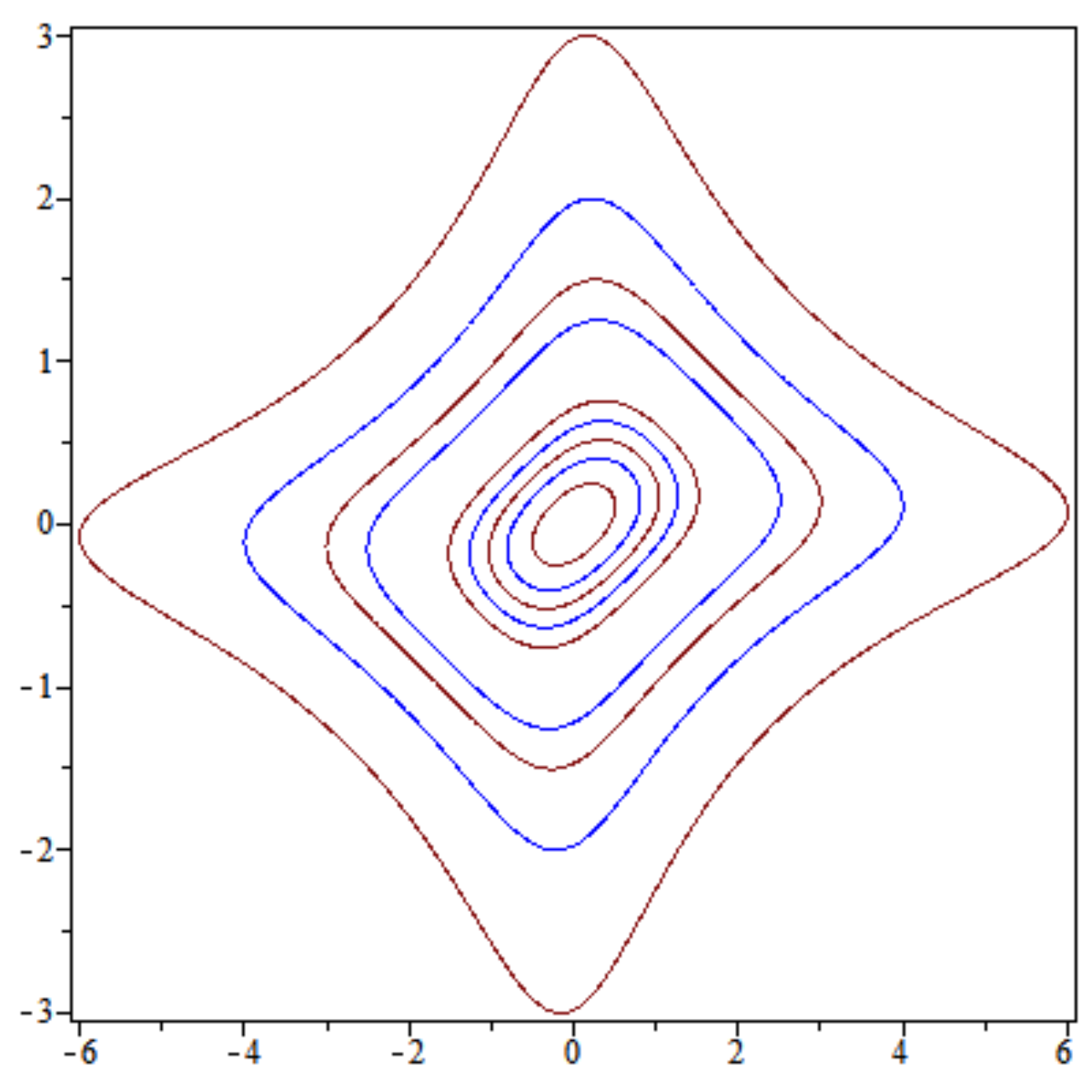}\hspace{0.5cm}
\includegraphics[scale=0.32]{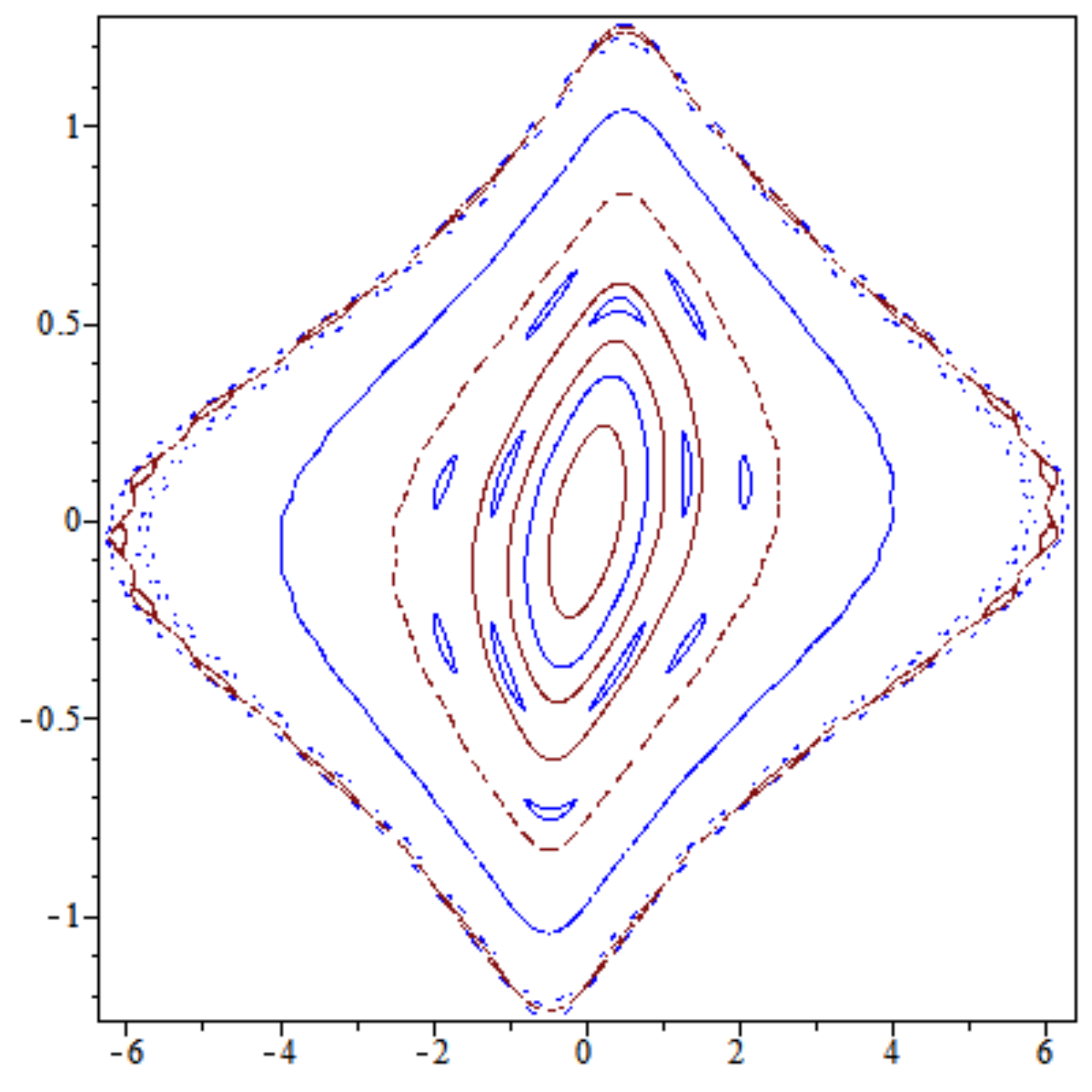}\hspace{0.5cm}
\includegraphics[scale=0.32]{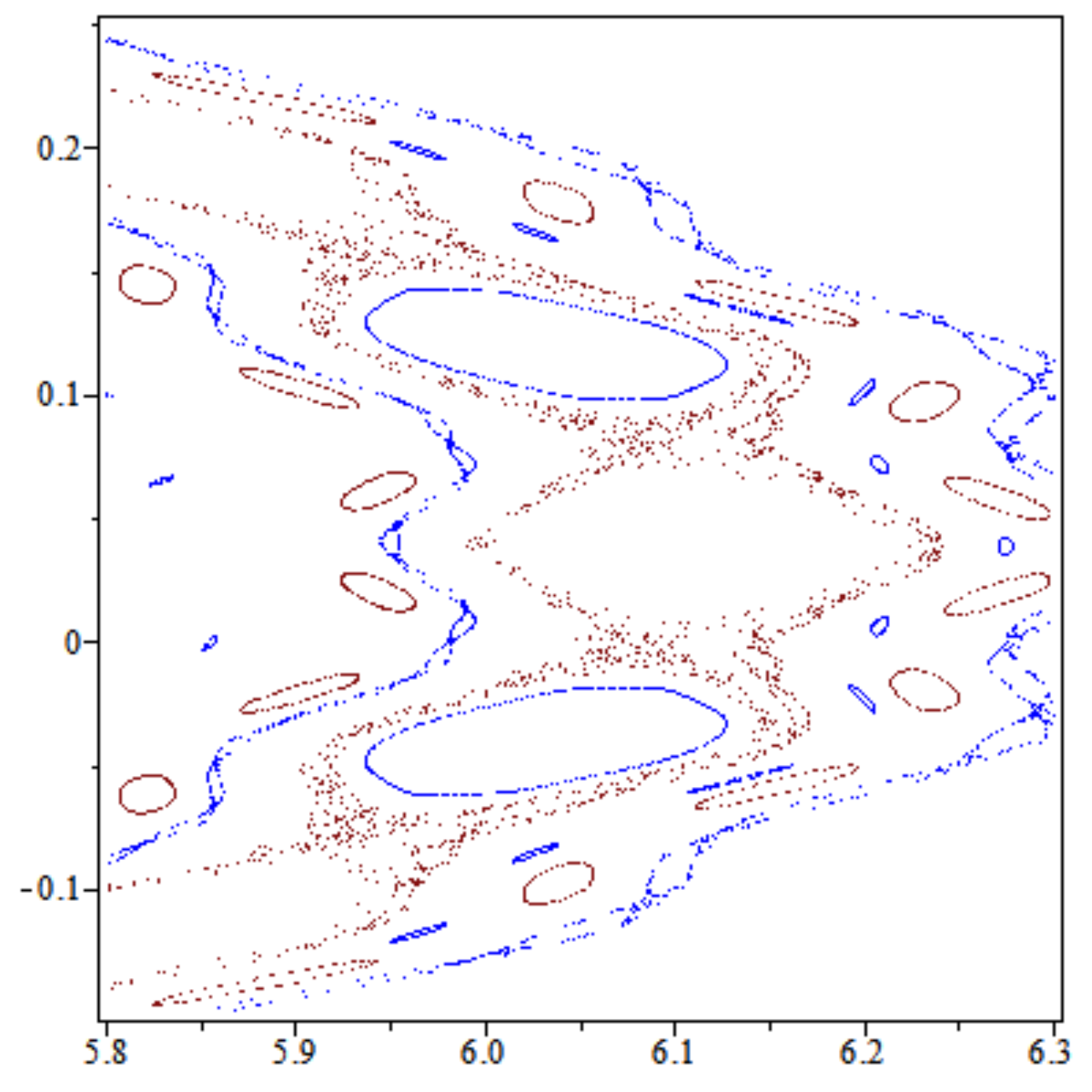}\\
\hspace{0.2cm} (a) \hspace{4.5cm} (b) \hspace{4.5cm} (c)\hspace{1cm}
\end{center}
\begin{center}
Figure 1: In (a) and (b), some orbits of the maps
$G_{{1}/{b},{1}/{a}}$ and $F_{b,a}$ with $a=2$, $b=1/2$ are
depicted. In (c), a detail of (b) is shown.
\end{center}

\section{The integrable case}\label{casintegrable}
Consider the non-autonomous recurrence
\begin{equation}\label{ab}
x_{n+2}=-x_{n}+\frac{ x_{n+1}}{\beta_n+x_{n+1}^2},
\end{equation}
with
\begin{equation*}
\beta_n\,=\,\left\{\begin{array}{lllr} a>0&{\mbox{for}}&n=2k,&\\
b>0&{\mbox{for}}&\,n=2k+1,&k\in\N,
\end{array}\right.
\end{equation*}
The local analysis of the constant recurrences generated  by it and
the level sets of its associated non-autonomous invariant $
I(x,y,n)=\beta_n x^2+ \beta_{n+1} y^2-xy+x^2y^2,$ has been already
done in \cite{CJK} (see also \cite{BR2} and \cite{GM}). In this
section we complete the description of the global dynamics of the
difference equation (\ref{ab}) and its associated composition map
$G_{b,a}$.

\subsection{Dynamics of the discrete dynamical system generated by $G_{b,a}$}\label{gba}

The above invariant for~\eqref{ab},  gives the first integral
$$V_{b,a}(x,y):=I(x,y,2)=a x^2+ b y^2-xy+x^2y^2,$$
for the the map $G_{b,a},$
\begin{equation}\label{mapgba}
G_{b,a}(x,y)=\left(-x+\frac{y}{a+y^2}, -y +\frac{(a+y^2)(y-x(a+y^2))}{b(a+y^2)^2+(y-x(a+y^2))^2}\right).
\end{equation}
In \cite{CJK}  a full description of the level sets of $V_{b,a}$ is
given, they are also described in Theorem \ref{rotacions}, bellow.
In a few words, the level sets are like (a) in Figure 1 when $ab\ge
1/4$ and like Figure 2 when $ab<1/4$.

\begin{center}
\includegraphics[scale=0.40]{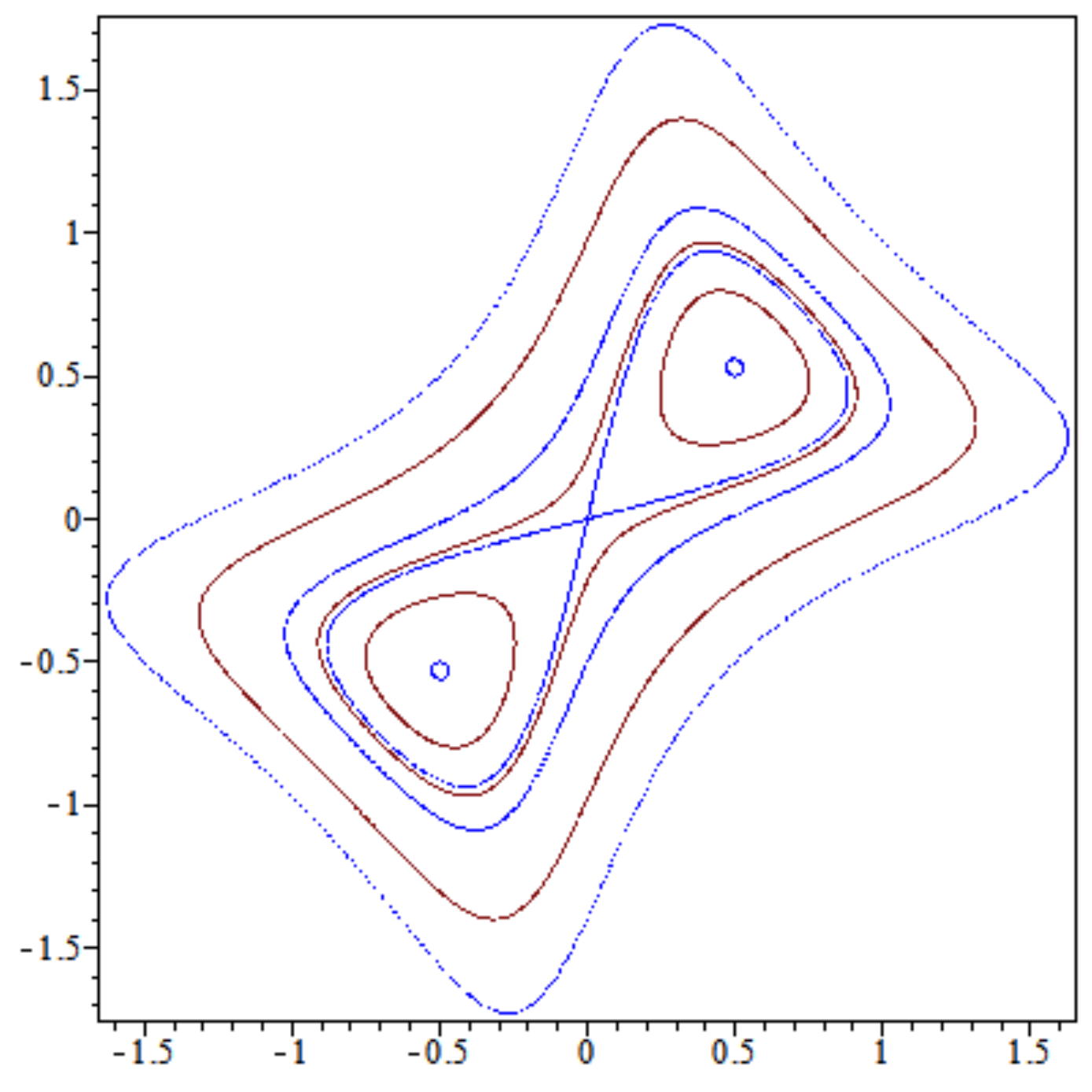}
\end{center}
\begin{center}
Figure 2: Some orbits of the maps $G_{b,a}$ with $a=1/4$, $b=2/9$.
\end{center}

To describe the dynamics on each level set we need to apply the
result bellow, which follows from \cite{CGM}.

\begin{propo}\label{propogen} Let $\phi:\U\to \U,$
$\U\subset \R^2$ be a smooth planar area preserving map having
finitely many fixed points and  a smooth first integral $V:\U\to
\R$. Then:
\begin{itemize}
\item[(a)] When a connected component of a  level set of $V$  does not contain fixed points of $\Phi,$ is invariant and
\begin{enumerate}
\item[(i)] it is diffeomorphic to a circle then $\Phi$, on it,  is conjugated to a rotation.
\item [(ii)] it  is diffeomorphic to the real line then $\Phi$, on it,  is conjugated to a translation.
\end{enumerate}
\item[(b)] When a connected component of a  level set contains fixed points of $\Phi$. Then on each invariant connected component of this level set
minus the fixed points,  $\Phi$ is conjugated to a translation.
\end{itemize}

\end{propo}

Taking into account the above proposition and the topology of the
level sets of $V_{b,a}$, we obtain the characterization of the
dynamics of the composition maps on these invariant level curves.
This result improves the ones obtained in~\cite{CJK} for $ab>0$,
where only the stability problem was considered.

We will use the following notations:
$$
P_{\pm}=(x_{\pm},y_{\pm}):=\left(\pm\sqrt{-b+\frac{1}{2}\sqrt{\frac{b}{a}}},\pm\sqrt{-a+\frac{1}{2}\sqrt{\frac{a}{b}}}\right).
$$
These points correspond to the non-zero fixed points of $G_{b,a}$, that only exist for $ab<1/4$. We also set
$h_{\min}:=V_{b,a}(P_{\pm})=-\frac{1}{4}-ab+\sqrt{ab}.$

\begin{teo}\label{rotacions}
\begin{itemize}
\item[(a)] If $ab\geq 1/4$, then the origin is the unique fixed point of $G_{b,a}$ and it is of elliptic type. The sets $\{V_{b,a}(x,y)=h>0\}$ surround the origin, are diffeomorphic to circles and invariant under $G_{b,a}$. Moreover  on each of them the map is conjugated to a rotation.
\item[(b)] If $ab< 1/4$, then:
\begin{itemize}
\item[(i)] The origin is a hyperbolic saddle of $G_{b,a}$. The set $\{V_{b,a}(x,y)=0\}\setminus\{(0,0)\}$  coincides with the stable and unstable manifolds of the saddle and consists of two
curves diffeomorfic to the real line. Each of them is invariant under $G_{b,a}$ and the map is conjugated to a translation.
\item[(ii)] The sets $\{V_{b,a}(x,y)=h>0\}$ are diffeomorphic
to  circles surrounding the set $\{V_{b,a}\leq 0\}$. They are invariant under $G_{b,a}$ and on each of these sets the map is conjugated to a
rotation.
\item[(iii)] $P_{\pm}$ are fixed points of $G_{b,a}$ of elliptic type.
\item[(iv)] The sets $\{V_{b,a}(x,y)=h<0,h\neq h_{\min}\}$ are diffeomorphic
to two circles surrounding $P_{+}$ and $P_{-}$ respectively. Each of these circles is invariant under $G_{b,a}$, and the map on each of these
sets is conjugated to a rotation of the circle.
\end{itemize}
\end{itemize}
\end{teo}

In next subsection we will use the above result  to describe the behavior of the sequences $\{x_n\}_{n\in\Z}$ generated by the recurrence~(\ref{ab}).

Notice that when the   level sets of $V_{b,a}$ are diffeomorphic to circles, the above result allows to associate to each level set $\{V_{b,a}(x,y)=h\}$ the rotation number  of the map $G_{b,a}$ restricted to it. We denote it by $\rho_{b,a}(h)$ and we call the function $\rho_{b,a}$, {\it the rotation number function}. Recall that when $\rho_{b,a}(h)=p/ q\in\Q,$ with $\gcd(p,q)=1$, then $\{V_{b,a}(x,y)=h\}$ is full of $q$-periodic points of $G_{b,a}$ and when  $\rho_{b,a}(h)\not\in\Q,$ then the orbit of $G_{b,a}$ is dense on $\{V_{b,a}(x,y)=h\}$.
 Notice that  $\rho_{b,a}(h)$ is a well defined function of $h$, because when  $\{V_{b,a}(x,y)=h\}$ has two ovals, then $\rho_{b,a}$ coincides on both because of the invariance of the map
 $G_{b,a}$ by the change of variables $(x,y)\longrightarrow(-x,-y)$. Hence, to describe the periods of $G_{b,a}$ we must study the map $\rho_{b,a}(h)$.

Following the techniques introduced in  \cite{CGM-ly,CGM} and some new tools we  prove the following result:

\begin{propo}\label{limit} (i) Let $\rho_{b,a}(h)$ be
the rotation number map  associated to $G_{b,a}$ for $h>0$. Then it is analytic on $(0,+\infty)$. Moreover, $\lim_{h\to +\infty}\rho_{b,a}(h)=1/2$ and
\begin{equation*}
\lim_{h\to 0^+}\rho_{b,a}(h)=\begin{cases} 0, &\mbox{ when } \quad ab<1/4,\\
\displaystyle{\frac{1}{2\pi}\arccos\left(\frac{1}{2ab}-1\right)},&\mbox{ when } \quad ab\ge1/4.
\end{cases}
\end{equation*}

(ii) When $ab<1/4,$ $\rho_{b,a}(h)$ is also defined for
$h\in(h_{\min},0)$ it is analytic on this interval and
\begin{equation*}
\lim_{h\to
P_{\pm}}\rho_{b,a}(h)=\frac{1}{2\pi}\,\arccos(1-16\,\sqrt{ab}+32
ab).
\end{equation*}
\end{propo}

\begin{proof}  (i) Following \cite[Thm 1]{CGM} under our hypotheses the rotation number function on each level set admits a dynamical interpretation in terms of the time of the flow associated to the planar vector field
$$
X_{b,a}(x,y)=\dps{\left(-\frac{\partial V_{b,a}}{\partial y},\frac{\partial V_{b,a}}{\partial x}\right)}=(x-2by-2x^2y,2ax-y+2xy^2).
$$
More concretely, each closed curve $\{V_{b,a}(x,y)=h\}$ is a periodic orbit of period $T(h)$ of the differential equation associated to $X_{b,a}$, and
moreover $G_{b,a}(x,y)=\varphi(\tau(h);x,y)$, where $\varphi$ is the flow associated to $X_{b,a}$. Then  $\rho_{b,a}(h)=\tau(h)/T(h)$. From this expression it is not difficult to see that it is an analytic function for $h\in(0,\infty)$, see \cite{CGM-ly}.

A straightforward computation shows that when $ab>1/4$
$$\mathrm{Spec}(DG_{b,a}(0,0))=\left\{\lambda^+,\lambda^- \right\},\quad \mbox{where}
\quad\lambda^\pm=\frac{1-2ab\pm i\,\sqrt{4ab-1}}{2ab}. $$ By using the tools
developed to prove \cite[Prop. 19]{CGM-ly}  we get that
$$\lim\limits_{h\to 0^+}\rho_{b,a}(h)=
\displaystyle{\frac{1}{2\pi}\arccos\left(\displaystyle{\mathrm{Re}(\lambda)}\right)}=\displaystyle{\frac{1}{2\pi}
\arccos\left(\frac{1}{2ab}-1\right)},$$
as we wanted to show.

To study the case $ab<1/4$ we will use another approach. Notice that in this situation
the level set ${\mathcal V}_0:=\{V_{b,a}(x,y)=0\}$, is formed by the origin, which is a saddle point, and their stable and unstable manifolds, which coincide, forming two loops, see Figure 2. Moreover on each of these loops the map $G_{b,a}$ is conjugated to a translation and
\[
\lim_{n\to\pm \infty} G^n_{b,a}(x_1,x_2)=(0,0),
\]
for each $(x_1,x_2)\in {\mathcal V}_0\setminus\{(0,0)\}$. Given any $\varepsilon>0,$ we take $N$ such that $1/(2N)<\varepsilon$
and the points $P^{\pm}_n:= G^{\pm n}_{b,a}(x_1,x_2)$. Observe that  both are in the same connected component of
 $\mathcal{V}_0\setminus\{(0,0)\}.$ Note that  $G^{2N}_{b,a}(P^-_n)=P^+_n$. By continuity of $G^{2N}_{b,a}$ we known
 that the image for $G^{2N}_{b,a}$ of a point, near $P_n^-$, on  $\{V_{b,a}(x,y)=h>0\}$ is, for $h>0$ small enough,
  as close as we want to $P_n^+.$ This implies that after $2N$ iterates this point has given less that half a turn to $\{V_{b,a}(x,y)=h>0\}$.
 Consequently, for  $h$ small enough, $\rho_{b,a}(h)<1/(2N)<\varepsilon$, as we wanted to prove.

The case $ab=1/4$ follows from the above two results, again by
continuity arguments.

Finally, to compute $\lim\limits_{h\to +\infty} \rho_{b,a}(h)$ we consider the change of
variables given by $(z,w)=(1/x,1/y)$ which conjugates the behavior at infinity of
$G_{b,a}$ with the behavior at the origin of the new map:
$$\widetilde{G}(z,w)=\left(\displaystyle{ -{\frac {z \left(
1+a{w}^{2}\right) }{1-zw+a{w}^{2}}}},\displaystyle{-\frac{p(z,w)\,w}{q(z,w)}}\right),$$
where
$$\begin{array}{l}
p(z,w)=1+2a{w}^{2}-2zw+b{z}^{2}+{a}^{2}{w}^{4}-2az{w}^{3}+2ab{z}^{2}{
w}^{2}+{z}^{2}{w}^{2}+{a}^{2}b{z}^{2}{w}^{4}, \\
q(z,w)=1+2a{w}^{2}-zw+b{z}^{2}+{a}^{2}{w}^{4}+2ab{z}^{2}{w}^{2}+
{a}^{2}z{ w}^{5}+{a}^{2}b{z}^{2}{w}^{4}-a{z}^{2}{w}^{4}.
\end{array}$$
Clearly  $\widetilde{G}(z,w)=(-z+O_2(z,w),-w+O_2(z,w))$ and so  $D\widetilde{G}(0,0)$ has the
eigenvalues $\pm1=\pm e^{\pi i}.$ Hence arguing as in the study of $G_{b,a}$ near
$(0,0)$ we obtain that
$$\lim\limits_{h\to \infty}\rho_{b,a}(h)=
\displaystyle{\frac{\pi}{2\pi}}=\displaystyle{\frac{1}{2}},$$
as we wanted to prove.

(ii) Follows by using similar arguments.
\end{proof}

\begin{nota}\label{noconstant}
(i) Applying the  method developed in the previous proposition to study $G_{b}$ we
would obtain $\lim\limits_{h\to+\infty}\rho_b(h)=1/4$. This  gives a new proof of
this fact which is essentially different of the one given in  \cite{BR2}.

(ii) It holds that
$$\displaystyle{\frac{1}{2\pi}\arccos\left(\frac{1}{2ab}-1\right)}
<\displaystyle{\frac{1}{2\pi}\arccos\left(-1\right)}=\displaystyle{\frac{1}{2}}.$$
Hence the map $\rho_{b,a}(h)$ can not be a constant function. In particular this
implies that there are no positive values $a$ and $b$ for which $G_b$ or $G_{b,a}$
are globally periodic.
\end{nota}

Before stating our main result on the periods appearing for $G_{b,a},$ we study in detail the 2-periodic points.

\begin{lem}\label{2p} For $a>0$ and $b>0$ the map $G_{b,a}$ has 2-periodic points if and only if $0<ab<1/16.$ Moreover they are given by the two ovals $\{V_{b,a}(x,y)=-ab\}$.
\end{lem}
\begin{proof} We have to solve the system $G_{b,a}(G_{b,a}(x,y))=(x,y)$, which is equivalent to the new one
$
G_{b,a}(x,y)=G_{b,a}^{-1}(x,y).
$ After some manipulations it gives
$$
\left\{
\begin{array}{l}
(2{x}^{2}y-x+2by)\, R(x,y)=0, \\
(2x{y}^{2}+2ax-y)\,R(x,y)=0,
\end{array}
\right.
$$
where $R(x,y):={x}^{2}{y}^{2}+a{x}^{2}-xy+b{y}^{2}+ab=V_{b,a}(x,y)+ab$.
Its solutions have to contain also the fixed points of the map. It is not difficult to check that the fixed points are  the real solutions of the system $\{2{x}^{2}y-x+2by=0,\,2x{y}^{2}+2ax-y=0 \}$. Hence the 2-periodic points coincide with the set $\mathcal{R}:=\{R(x,y)=0\}$. Therefore we need to study when it has real points. Notice $R$ is a quadratic polynomial with respect to $x$. Its discriminant is
$$D(y):=-4b{y}^{4}+ \left(1 -8ab\right){y}^{2}-4b{a}^{2}.$$
Since $a>0$ and $b>0$ then $\mathcal R$ will be non-empty when $D(y)$ takes non-negative values. The discriminant of $D(y)$, with respect to $y^2$, is $16ab-1>0$. So when $0<ab\le 1/16$, the set $\mathcal R$ is non empty. When $ab=1/16$ it coincides with the other fixed points $P_\pm.$ For $ab<1/16$
it is formed by two ovals surrounding them, which form precisely the level set with rotation number $1/2$, that is $\rho_{b,a}(-ab)=1/2$. Hence the result follows.
\end{proof}

Recall that the origin is always a fixed point of the map  $G_{b,a}$ and that when
$ab<1/4$ it has two more fixed points. Next result gives the periodic points of
this map with period greater than one.

\begin{teo}\label{peri} Let $G_{b,a}$ the map defined in \eqref{mapgba} for $a>0$ and $b>0$.
\begin{enumerate}[(a)]
\item For $0<ab<1/16$ it has  continua of periodic points of all the
periods greater than one.
\item For $1/16\le ab\le1/4$ it has  continua of periodic points of all
the periods greater than two and has no periodic points of period two.
\item For $ab> 1/4$ it has  continua of periodic points of all the
periods $q$ such that there exists $p\in\N$, with $\gcd(p,q)=1$ and such that
\[
\frac pq
\in\left(\frac{1}{2\pi}\arccos\left(\frac{1}{2ab}-1\right),\dfrac 12
\right).
\]
In particular, there is a computable number $q_0(a,b)$ such that $G_{b,a}$ has continua of periodic points of all periods greater than $q_0(a,b)$.
\end{enumerate}
\end{teo}

\begin{proof}
It is a direct consequence of Theorem~\ref{rotacions},  Proposition~\ref{limit} and Lemma~\ref{2p}. Notice that the fact that all periods greater than 2 appear when $ab\le 1/4$ is because $q\ge 3,$ $1/q\in(0,1/2)$.
The final conclusion in (c) follows from the well-known fact that any open interval contains irreducible fractions with all the denominators
bigger that a given number which depends on their extremes. A constructive upper bound for this number is developed  in \cite{CGM-ly}.
\end{proof}

\begin{nota} (a) If   the rotation number function $\rho_{b,a}(h)$ was monotonous with respect to $h$ then we could also completely characterize  all the periods for the case $ab>1/4$. They would be precisely the ones given in item (c).

(b) Theorem~\ref{peri}  can be
easily adapted to obtain similar results for the  Gumov\-ski-Mira maps,
$G_b, b>0$. The corresponding results complete the ones of \cite{BR2} obtained for
the case $b\ge1/2$, where all the level sets of the corresponding first integral,
but the origin, are diffeomorphic to circles.
\end{nota}

\subsection{Study of the recurrence~(\ref{ab})}\label{recu-g}

Theorem \ref{rotacions} will allow us to characterize the sequences
$\{x_n\}_{n\in\N}$ generated by the recurrence~(\ref{ab}). Observe
that in fact since the map $G_{b,a}$ is a diffeomorphism of $\R^2$,
by using its inverse we can define the extended sequence
$\{x_n\}_{n\in\Z}$. Recall that it is said that a recurrence is {\sl
persistent} if any sequence $\{x_n\}_{n\in\Z}$ is contained in a
compact set or $\R$, which depends on the initial conditions $x_1$
and $x_2$.

Clearly for any $a>0$, $b>0$ and $x_1=x_2=0$ we get that $x_n=0$ for all $n\in\Z$. Next result describes the behaviour of  $\{x_n\}_{n\in\Z}$
for non-null initial conditions.

\begin{teo}\label{main} The recurrence (\ref{ab}) is persistent. Furthermore, let $\{x_n\}_{n\in\Z}$ be any sequence defined by (\ref{ab}) with initial conditions $x_1$
and $x_2$, $(x_1,x_2)\ne(0,0)$ and set
$$h_{\min}:=-\frac{1}{4}-ab+\sqrt{ab}\,\mbox{ and }\,
h_{+}:=\frac{1}{2}\left[-({a}^{2}+{b}^{2})-\frac{1}{2}+ \,\left( a+b \right) \,\sqrt { (a-b)^{2}+1 }\right].
$$ It holds that  $h_{\min}<h_+<0$.

\begin{itemize}

\item[(a)] Assume that $ab< 1/4$. Then:
\begin{itemize}
\item[(i)] For $(x_1,x_2)=P_{\pm}$, which coincide with the set  $\{V_{b,a}(x,y)=h_{\min}\}$, the
sequence $\{x_n\}_{n\in\Z}$ is two periodic when $a\ne b$ and a constant sequence when $a=b$.
\item[(ii)] For $(x_1,x_2)\neq (0,0)$ and $V_{b,a}(x_1,x_2)=0$, $\lim\limits_{n\to \pm\infty} x_n=0$.
\item[(iii)]  For a numerable and dense set of values of $h$, the level sets  $\{V_{b,a}(x,y)=h\}$ give initial conditions such that $\{x_n\}_{n\in\Z}$ is
a periodic sequence of even period which only depends on $h$.
\item[(iv)]  For the rest of values of $h$ (which  form a dense set of full measure),
the level sets  $\{V_{b,a}(x,y)=h\}$ give initial conditions such
that the adherence of the corresponding sequence $\{x_n\}_{n\in\Z}$
is given by one or two intervals. More concretely, this adherence is
given by one interval when either $a=b$; $h>0$; or $h_+\le h <0$ and
by two intervals when $a\neq b$ and $h_{\min}<h<h_+$.
\end{itemize}
\item[(b)] Assume that $ab\ge 1/4$. Then only the behaviors given in items (iii) and (iv) appear. Moreover in item (iv) the adherence is always
formed by a single interval.

\end{itemize}
\end{teo}

\begin{proof}
We will use the techniques introduced in~\cite{CGM11a} where we
studied the 2-periodic Lyness recurrences.
 First notice that the relation between the terms of the recurrence and the iterates of the composition map is given by
$$
G_{b,a}(x_{2n-1},x_{2n})=(x_{2n+1},x_{2n+2}), \quad G_{a,b}(x_{2n},x_{2n+1})=(x_{2n+2},x_{2n+3}),
$$
where $(x_1,x_2)\in\R^2$ and $n\in\Z$. Also observe that a simple computation shows that
$$V_{b,a}(x,y)=V_{a,b}(G_a(x,y)),$$ which implies that the odd terms of
the sequence $\{x_n\}_{n\in\Z}$ are contained in the projection on
the $x$-axis of the level sets of the form
$\{V_{b,a}(x,y)=V_{b,a}(x_1,x_2)=h\}$ and the even ones are in the
corresponding projection of the level sets
$\{V_{a,b}(x,y)=V_{a,b}(G_{a}(x_1,x_2))=h\}$.

The result (i) is clear. Statement (ii) is a new consequence of
Theorem \ref{rotacions} and the fact that the sets
$\{V_{b,a}(x,y)=0\}\setminus\{(0,0)\}$ are two homoclinic loops
starting and ending at $(0,0)$. To prove (iii) and (iv) we consider
the cases (a) and (b) together.

Again from Theorem \ref{rotacions} we know that the sequence
starting at a point $(x_1,x_2)$ such that
$\rho_{b,a}(V_{b,a}(x_1,x_2))=p/q\in\Q$  with $\gcd(p,q)=1$, will be
$2q$-periodic. When it holds that
$\rho_{b,a}(V_{b,a}(x_1,x_2))\not\in\Q,$ then it will  densely fill
the projection of $\{V_{b,a}(x,y)=h\}$ (the odd terms) together with
the projection of $\{V_{a,b}(x,y)=h\}$ (the even terms).

Hence to end the proof we only need to show that the rational values of   $\rho_{a,b}(V_{b,a}(h))$ are taken only for a numerable dense set of values $h$ and, when $\rho_{a,b}(V_{b,a}(h))\not\in\Q$, to distinguish whether the above to projections do or do not overlap.

Let us prove first that $\rho_{a,b}(V_{b,a}(h)\in\Q$ only for a
numerable and dense set of values of~$h$. First recall that from
Proposition \ref{limit} and Remark \ref{noconstant} we know that the
function is analytic and nonconstant in the intervals where it is
defined. Next we use the well-known fact that a nonconstant analytic
function has a numerable number of zeroes. This implies that for any
$r\in\Q$ the set $\mathcal{P}_{r}:=\{h\in\R\,:\,
\rho_{b,a}(h)-r=0\}$ is numerable. Therefore $\cup_{r\in\Q}
\mathcal{P}_r$ is numerable as we wanted to see. That it is dense
follows from the continuity of all the involved functions.

To end the proof we only need to decide whether the projections into
the $x$-axis of $\{V_{b,a}(x,y)=h\}$ and $\{V_{a,b}(x,y)=h\}$
intersect or not.

 When $h>0$, for any positive $a$ and $b$ both ovals $\{V_{b,a}(x,y)=h\}$ and $\{V_{a,b}(x,y)=h\}$ surround the origin. Hence the adherence of $\{x_n\}_{n\in\Z}$
is always one closed interval.

Assume now that $h<0$ (and so $ab<1/4$). An easy computation shows that when $a\neq b$, $G_{a}(P_{\pm})$ is neither $P_{\pm}$ nor $P_{\mp}$, so the points $P_{\pm}$ give rise to two different $2$-periodic solutions of the recurrence. When $a=b$, $P_{\pm}$ are fixed points of $G_a$, and the situation is like the previous one. So in the following we also assume that $a\neq b$.

 Since $G_{a}(P_{\pm})$ is neither $P_{\pm}$ nor $P_{\mp}$, and $G_a$ is a diffeomorphism, it is possible to
take ovals of the form $\{V_{b,a}(x,y)=h\}$ with $h_{\min}\simeq
h<0$ whose projections into the abscissa axis do not overlap with
the corresponding ones of $\{V_{a,b}(x,y)=h\}$. However, by
increasing $h$ the projections will start to overlap.

Let be $I^{\pm}=I^{\pm}(a,b,h)$, and $J^{\pm}=J^{\pm}(a,b,h)$ be the
projections on the abscissa axis of the ovals contained in the level
sets given by $\{V_{b,a}(x,y)=h\}$ and $\{V_{a,b}(x,y)=h\}$
respectively, surrounding $P_+$ and $P_-$. We want to detect the
values of $h$ for which the intervals $I^+$ and $J^+$ (and $I^-$ and
$J^-$ respectively) have exactly one common point. First we seek for
their boundaries. Since the level sets are given by quartic curves,
quadratic with respect the $y$-variable, these points will
correspond with values of $x$ for which the discriminant of the
quadratic equation with respect to $y$ is zero. So, we compute
$$\begin{array}{l}
R_1(x,h,a,b):=\mbox{dis}\,(V_{b,a}(x,y)-h,y)=-4\,a\,{x}^{4}+ \left(
-4\,ab+4\,h+1 \right) {x}^{2}+4\,b\,h,\\
R_2(x,h,a,b):=\mbox{dis}\,(V_{a,b}(x,y)-h,y)=-4\,b\,{x}^{4}+ \left( -4\,ab+4\,h+1 \right) {x}^{2}+4\,a\,h.
\end{array}
$$
To search for relations among $a,b$ and $h$ for which these functions have some common solution, $x$, we compute the following resultant
$$\mathrm{Res}(R_1,R_2;x)=
256 \left( a-b \right) ^{4}{h}^{2} \left(16\,{h}^{2}+ \left( 16\,{a}^{2}+16\,{b}^{2}+8 \right) h+ \left( 4\,ba-1 \right) ^{2} \right) ^{2}$$
Since $a\neq b$,  the possible bifurcation values of $h$ are given by the roots of the quadratic polynomial
\begin{equation}\label{polinomi}
P_1(h):=16\,{h}^{2}+ \left( 16\,{a}^{2}+16\,{b}^{2}+8 \right) h+ \left( 4\,ba-1 \right) ^{2},
\end{equation}
which are
$$
h_{\pm}:=\frac{1}{2}\left[-({a}^{2}+{b}^{2})-\frac{1}{2}\pm \,\left( a+b \right) \,\sqrt { (a-b)^{2}+1 }\right].
$$
Then the number of disjoint intervals depends only on the relative positions of $h_{\min},h_{\pm}$ and $V_{b,a}(x_1,x_2)$.

To study the relative positions of $h_{\min}$ and $h_{\pm}$, we introduce the auxiliary polynomial $P_2(h)=(h+ab+\frac{1}{4})^2-ab,$ because one of
its roots is $h_{\min}$. Then
$$
\mathrm{Res}(P_1,P_2;h)=16\,(4ab-1)^2\,(a-b)^2.
$$
This means that we can split the parameter set
$\mathcal{A}:=\{(a,b),\,a,b>0,\,ab< 1/4\}$ into the following
disjoints sets:
$$
\mathcal{A}=\{a,b>0,\,ab< 1/4,a>b\}\cup \{a,b>0,\,ab< 1/4,a=b\}\cup\{a,b>0,\,ab< 1/4,a<b\}.
$$
Observe that, because of the properties of the resultant,  the relative positions of $h_{\min}$ and $h_{\pm}$ on the above subsets do not vary. So testing their values for some particular
parameters, we obtain that when $a\neq b$ we always have that $h_{\min}<h_+<0$ (we also notice that when $a=b$, $h_{-}<h_{\min}=h_{+}$).

In summary, the above considerations imply that when $a\neq b$, and since $G_a(P_{\pm})\neq P_{\pm}$, the projections into the abscissa axis of
all the ovals $\{V_{b,a}(x,y)=h\}$ and $\{V_{a,b}=h\}$ when $h_{\min}< h<h_+$ do not overlap, giving rise to two different intervals. When
$h_+\leq h<0$ they overlap giving only one interval.
\end{proof}

From the above theorem we get that the list of behaviours (i)-(iv)
listed in the introductions are the only ones that the sequences
generated by (\ref{ab}) can present. Moreover, given some values of
$a$ and $b$ the information about the possible periods of the
periodic sequences  can be obtained from Theorem \ref{peri}.

We end this subsection with a concrete example.

\begin{example}
Consider  $a=0.01$, $b=0.49$. For any initial condition $(x_1,x_2)$, set $h_0=V_{b,a}(x_1,x_2)$. From the
definitions given in Theorem \ref{main} we have that $h_{\min}=-0.1849$ and $h_{+}\simeq -0.0928$, so
\begin{itemize}
\item If $h_0\in(h_{\min},h_{+})$, the sequence generated by the recurrence (\ref{ab}) is either periodic (with even period) or it fills densely two intervals.
\item If $h_0\in[h_{+},0)$, the sequence generated by the recurrence (\ref{ab}) is either periodic (with period a multiple of $2$) or it fills densely one interval.
\end{itemize}

For instance if we consider the initial conditions
$(x_1,x_2)=(x_+-0.1,y_+-0.1)$, we have that
$h_0=H_{b,a}(x_1,x_2)\simeq -0.1459$, hence
$h_0\in(h_{\min},h_{+})$, and the corresponding sequence generated
by the recurrence (\ref{ab}) is either periodic or it densely fills
two disjoint intervals. This last option is the one that appears in
our numerical simulations, see Figure 3.
\begin{center}
\includegraphics[scale=0.40]{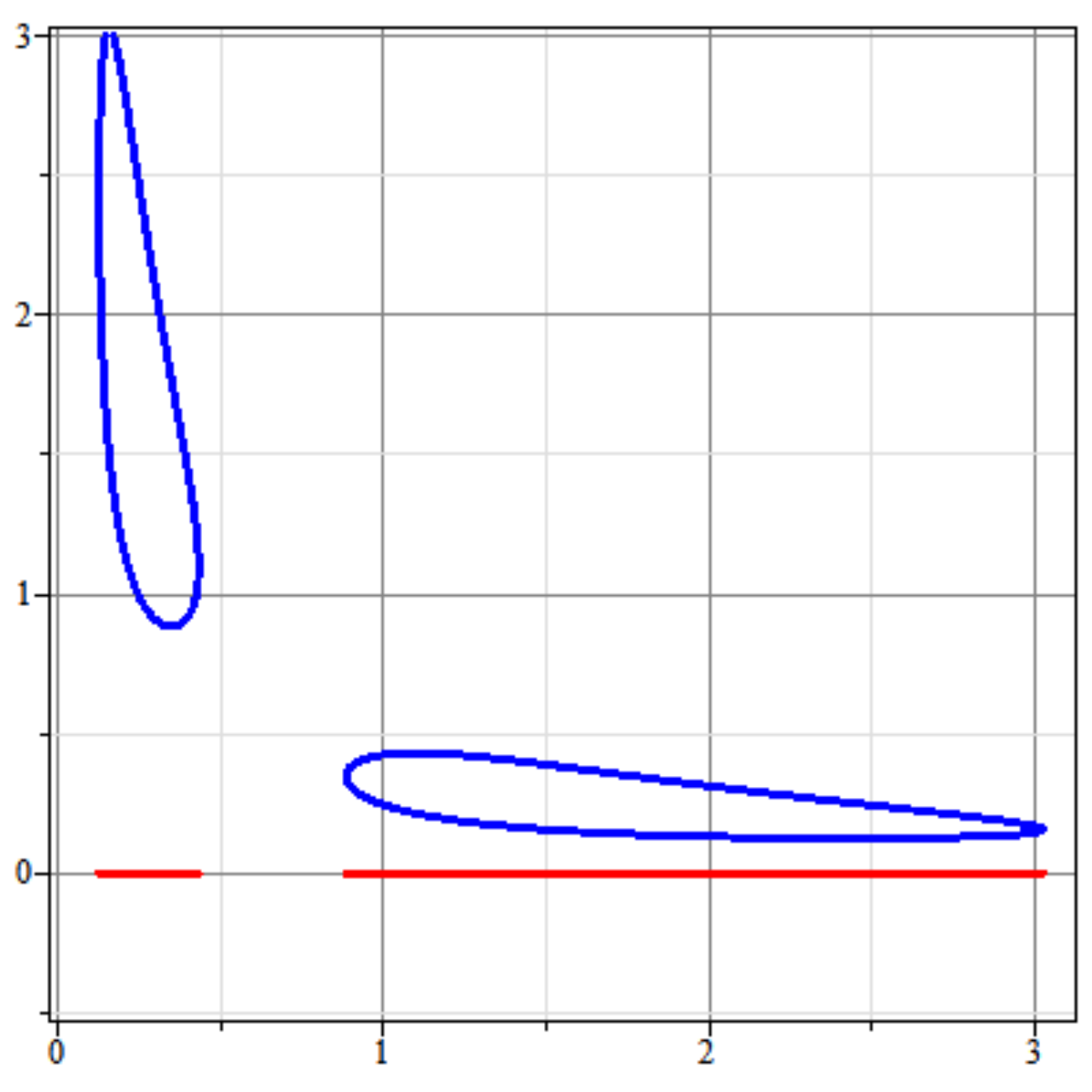}
\end{center}
\begin{center}
Figure 3: Here $a=0.01$, $b=0.49$. The ovals $\{V_{b,a}(x,y)=h_0\}$
and $\{V_{a,b}=h_0\}$ with $h_0\simeq \ -0.1459$ and their
projections on the abscissa axis.
\end{center}
\end{example}

\subsection{Integrability for general $k$-periodic Gumovski-Mira recurrences}

In this section we search for invariants of the form
\begin{equation}\label{formainvariant}
I(x,y,n)=\sum_{i+j=1}^4 I^{i,j}_{n} x^iy^j
\end{equation}
for  the recurrence (\ref{ab}),
but for any sequence $\{\beta_n\}_n$ of positive numbers.  This form is natural because includes the invariants known for $\{\beta_n\}_n$ constant or 2-periodic. As far as we know, this method was introduced  in \cite{FJL}. It was also
 used in \cite{CGM11b} to prove that non-autonomous Lyness recurrences, $x_{n+2}=(\alpha_n+x_{n+1})/x_n$, only admit invariants of a suitable form when  $\{\alpha_n\}_n$ are periodic with periods 1,2,3 and 6. If we could find a new integrable Gumovski-Mira recurrence all the methods used in this section could be adapted to study it. Unfortunately, next result shows that  for the Gumovsky-Mira non-autonomous recurrences, only the cases of $\{\beta_n\}_n$ with period 1 and 2 seem to be integrable. We have chosen  $I(x,y,n)$ with terms of degree at most 4 for simplicity. Bigger degrees could be also studied by using the same approach.

\begin{propo}\label{noautonom} The non-autonomous $k$-periodic recurrence (\ref{mirainte}), with  $\{\beta_n\}_n$ a sequence of positive numbers,
has invariants of the form (\ref{formainvariant}) if and only if $\{\beta_n\}_n$ is $k$-periodic and $k\in\{1,2\}$. Moreover, in these cases the invariants  are $I(x,y,n)=\beta_n x^2+ \beta_{n+1}
y^2+\gamma x^2y^2-\gamma xy$, for any $\gamma\ne0$. \end{propo}
\begin{proof} The condition that a function $I(x,y,n)$ of the form (\ref{formainvariant}) is a non-autonomous invariant of the recurrence (\ref{ab})
writes as
$$
I\left(y,-x+\frac{y}{\beta_n+y^2},n+1\right)-I(x,y,n)=0,
$$
for all $(x,y)\in\R^2$ and all $n\in\mathbb{N}$. Imposing that all
the monomials of the left hand side of the above relation vanish and
after some work we get in particular that $\beta_{n+2}=\beta_n$ and
so that such an invariant exists only when $\{\beta_n\}_n$ is
constant or 2-periodic. For both cases we also obtain the explicit
form of the known invariants.
\end{proof}

\section{The chaotic case}\label{se:3}
In this section  we consider the non-autonomous recurrence
\begin{equation}\label{Fab}
x_{n+2}=-x_{n}+\frac{\alpha_n x_{n+1}}{1+x_{n+1}^2},
\end{equation}
with
\begin{equation*}
\alpha_n\,=\,\left\{\begin{array}{lllr} a>0&{\mbox{for}}&n=2k,&\\
b>0&{\mbox{for}}&\,n=2k+1,&k\in\N.
\end{array}\right.
\end{equation*}

Recall that for each fixed value of $a$, the map $F_a$ associated to the corresponding autonomous recurrence
\begin{equation}\label{aut}
x_{n+2}=-x_{n}+\frac{a x_{n+1}}{1+x_{n+1}^2},
\end{equation}
has the first integral $$W_a(x,y)=x^2y^2+ x^2+y^2-a xy.$$ Moreover
it is conjugated to $G_{1/a}$. Therefore  $F_a$ and $F_{a,a}$ are
integrable and the dynamical systems generated by them as well as
the recurrences generated by (\ref{aut}) can be easily described
using the results of the previous section.

\subsection{The map $F_{b,a}$ as a perturbed twist}

Here we will see that the map $F_{b,a}$, which writes as,
 \begin{equation}\label{mapfba}
F_{b,a}(x,y)=\left(-x+\frac{ay}{1+y^2}, -y +\frac{b(1+y^2)(ay-x(1+y^2))}{(1+y^2)^2+(ay-x(1+y^2))^2}\right),
\end{equation}
is far from preserving neither the integrability nor the
translation-like dynamics, proved for $G_{b,a}$. In fact, our
numerical simulations show  all the features of a non-integrable
perturbed twist map, that is: many invariant curves and, between
them, couples of orbits of $p$-periodic points (for several values
of $p$), half of them of elliptic type and the other half of
hyperbolic saddle type, see for instance  \cite[Chapter 6]{AP}.
Notice that the sequences corresponding to the stable manifolds of
the hyperbolic saddles are precisely the ones that give rise to the
behaviour (v) described in the introduction.

More precisely, for instance, when $a \in (0,2)$ the positive level sets of $W_a$ in $\R^2$ are closed curves surrounding the origin, which is
the unique fixed point of $F_{a}$ (see Figure 1 (a)) and we know that the action of $F_{a}$
restricted to each of them is conjugated to a rotation. In fact, both $F_{a}$ and $F_a^2$ can be
seen as area preserving integrable twist diffeomorphisms since, at least locally, the rotation number function associated $F_a$ to most level
set of $W_a$ is monotonous. Therefore when we consider $F_{a+\varepsilon,a}$, for $\varepsilon$ small enough, it can be thought as a
perturbed twist map. Our numerical experiments show that this perturbed twist map is no more integrable and
that this behaviour remains even when the values of $b$ are not close to
the ones of $a$. This is the situation depicted in Figure 1 (b) and (c). Unfortunately we have not been able to prove the non-integrability of $F_{b,a}$ for $b\ne a.$

When $a>2$ the negative level sets of $W_a$ are two nests of closed
curves surrounding two fixed points of $F_{a}$. These two nests are
surrounded by a polycycle composed by a saddle point of the map and
a couple of homoclinic curves surrounding both elliptic fixed
points. The rest of the level sets are given by closed curves
surrounding the polycycle, see Figure~2. The dynamics of $F_a$ on
the invariant level sets of $W_a$ is again translation-like. In this
case, when we consider $b=a+\varepsilon$ for $\varepsilon$ small
enough, near the nests of $F_a$ the map $F_{b,a}$ can be seen again
as a non-integrable perturbation of the twist map $F_a^2=F_{a,a}$,
and the same phenomena than above  appear. These behaviours are also
observable when the values of $b$ are not close to the ones of $a$.
This is the situation exemplified in Figure 4, where it can also be
observed orbits of $F_{b,a}$ which seem to fill densely a  sets of
positive Lebesgue measure in $\R^2$.

\begin{center}
\includegraphics[scale=0.45]{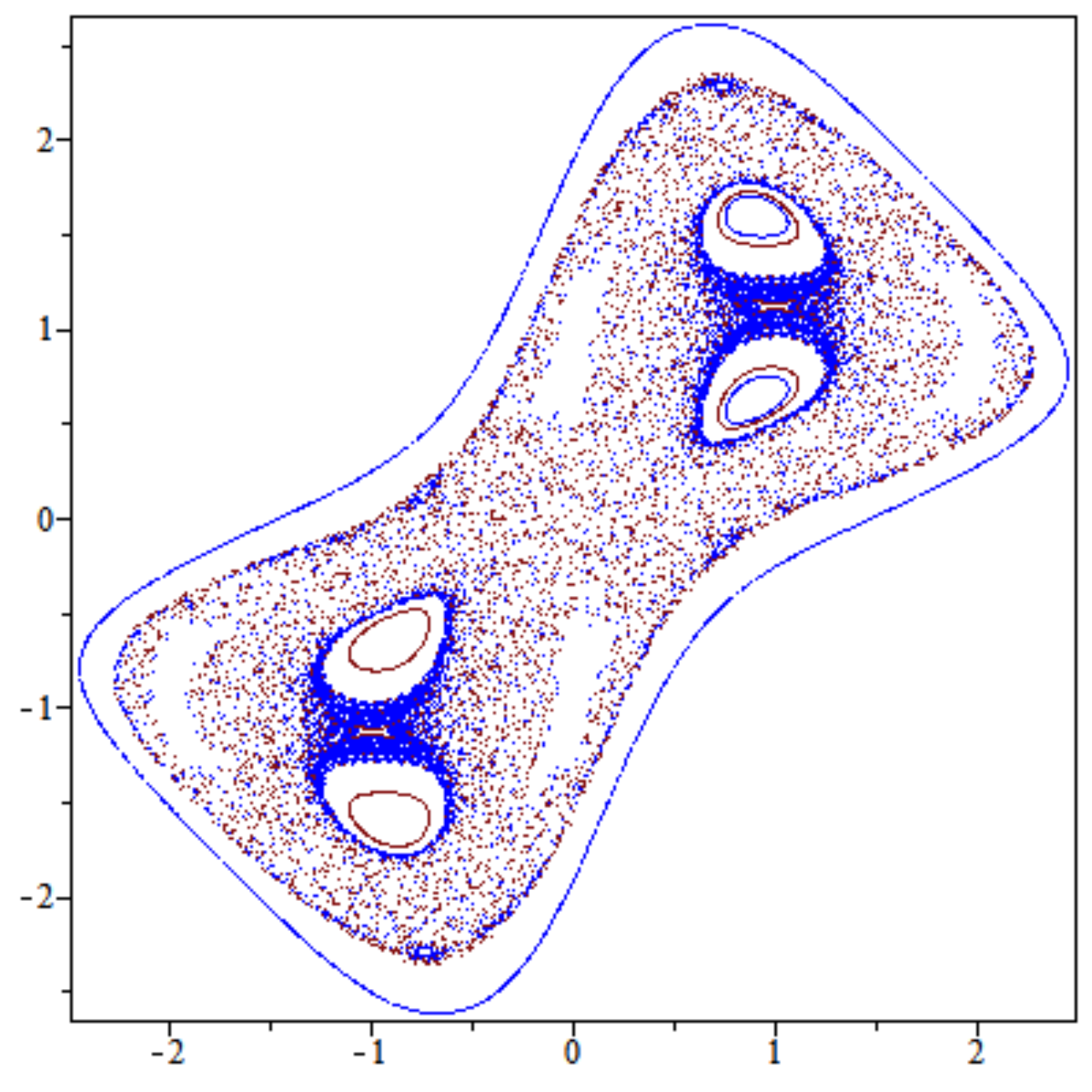}\hspace{0.5cm}
\includegraphics[scale=0.45]{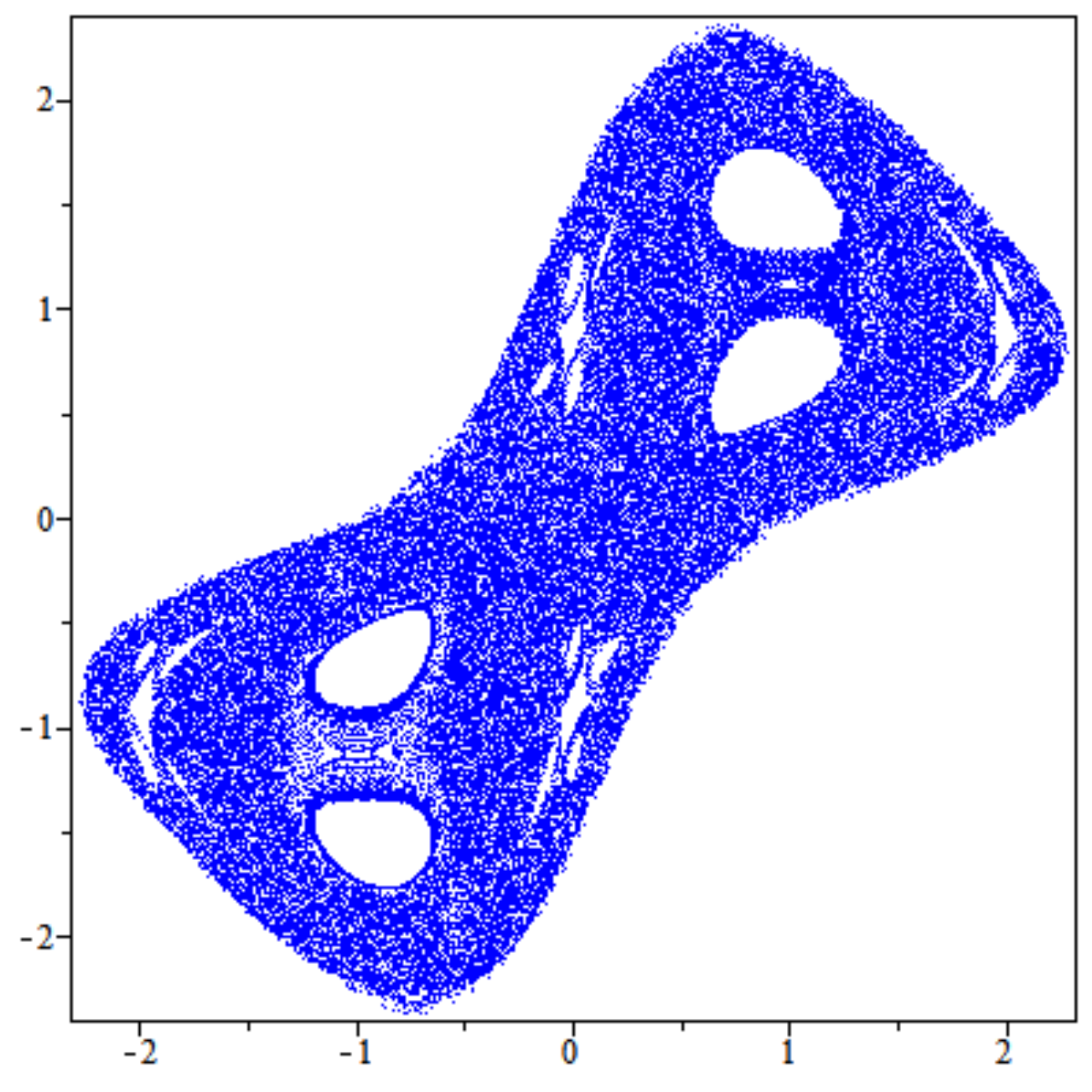}\\
\hspace{0.4cm} (a) \hspace{6.2cm} (b) \hspace{3.4cm}
\end{center}
\begin{center}
Figure 4: (a) Several orbits of the map $F_{9/2,4}.$  (b) Plot of $10^5$ iterates of the orbit of  $F_{9/2,4}$ starting at
 $(1.18,0.1)$.
\end{center}

From the viewpoint of the solutions $\{x_n\}_{n\in\Z}$ of the equation (\ref{Fab}), another difference is that the number of disjoint
intervals of the adherence of an orbit can be more than $2$, as it is shown in the following examples.\newline

\begin{example}\label{exa}
Consider  $a=2$ and $b=1/2$. We give some  initial conditions giving rise to
solutions $\{x_n\}_{n\in\Z}$ which, numerically, seems to densely fill more than 2 intervals.
\begin{center}
\begin{tabular}{|c|c|}
\hline
Initial conditions & Intervals \\
\hline
$(1.10,0.5)$ & 7 \\
$(1.25,0.5)$ & 16 \\
$(1.29,0.5)$ & 1 \\
$(1.30,0.5)$ & 23 \\
$(1.35,0.5)$ & 1 \\
$(1.40,0.5)$ & 6 \\
$(1.48,0.5)$ & 6 \\
\hline
\end{tabular}
\end{center}
In Figure 5 are depicted the $6$ intervals corresponding to the initial condition $(1.48,0.5)$, and the $16$ intervals of the initial condition
$(1.25,0.5)$
\end{example}

\begin{center}
\includegraphics[scale=0.45]{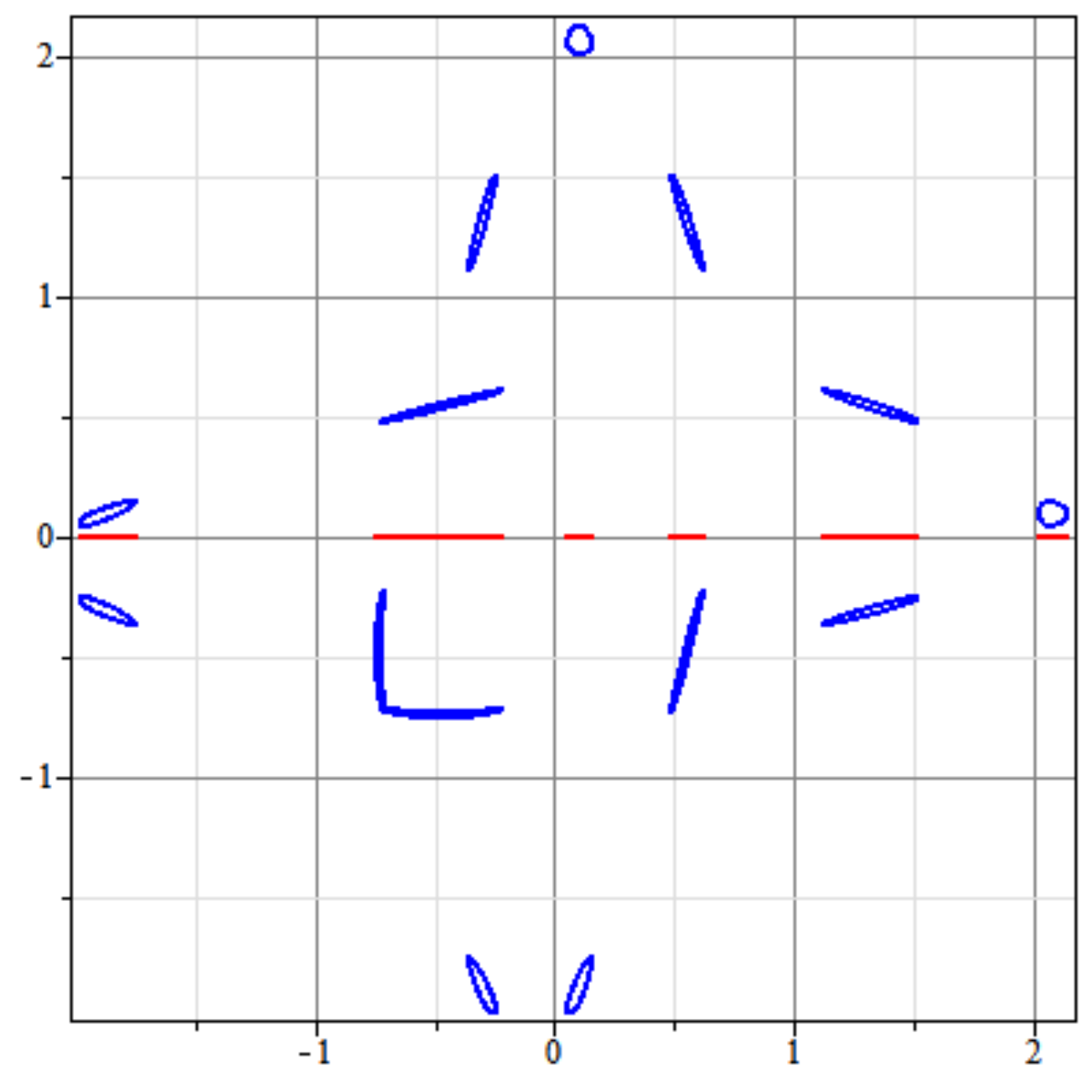}\hspace{0.5cm}
\includegraphics[scale=0.45]{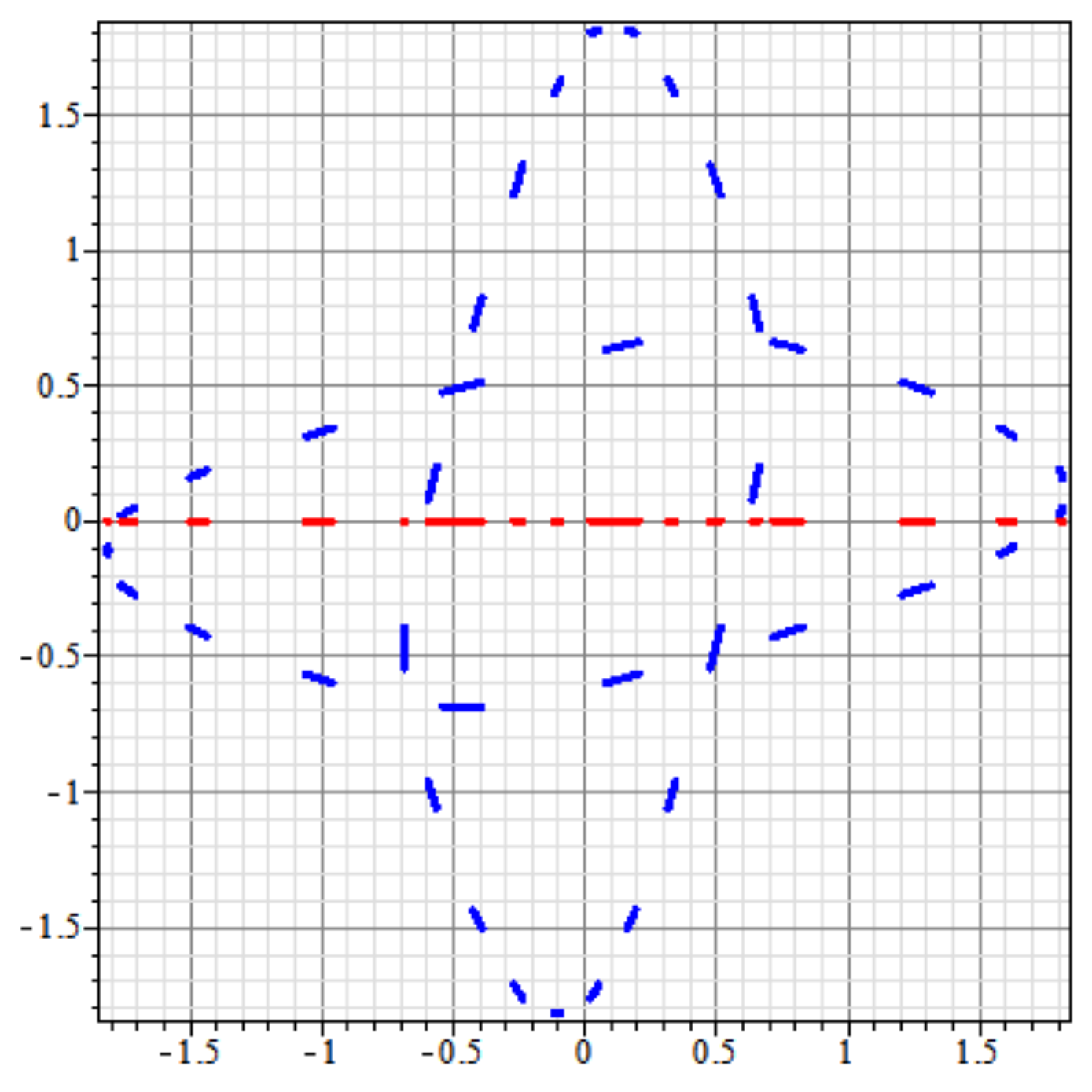}\\
\hspace{0.4cm} (a) \hspace{6.2cm} (b) \hspace{3.4cm}
\end{center}
\begin{center}
Figure 5: Orbits of  $F_{1/2,2}$ and their images under  $F_2$
corresponding to the initial conditions $(1.48,0.5)$ and
$(1.25,0.5)$ respectively, together with their corresponding
projections on the abscissa axis.
\end{center}
\fiexemple

\begin{example} An example similar to  Example \ref{exa} with $a=4$ and $b=9/2$.
\begin{center}
\begin{tabular}{|c|c|}
\hline
Initial conditions & Intervals \\
\hline
$(2.5,0.5)$ & 4\\
$(2.9,0.5)$ & 1\\
$(3.0,0.5)$ & 5 \\
$(3.1,0.5)$ & 3 \\
$(3.2,0.5)$ & 1 \\
$(3.6,0.5)$ & 4 \\
\hline
\end{tabular}
\end{center}
In Figure 6 we show the $4$ intervals corresponding to the initial condition $(2.5,0.5)$, and the $5$ intervals of the initial condition
$(3,0.5)$.
\end{example}

\begin{center}
\includegraphics[scale=0.45]{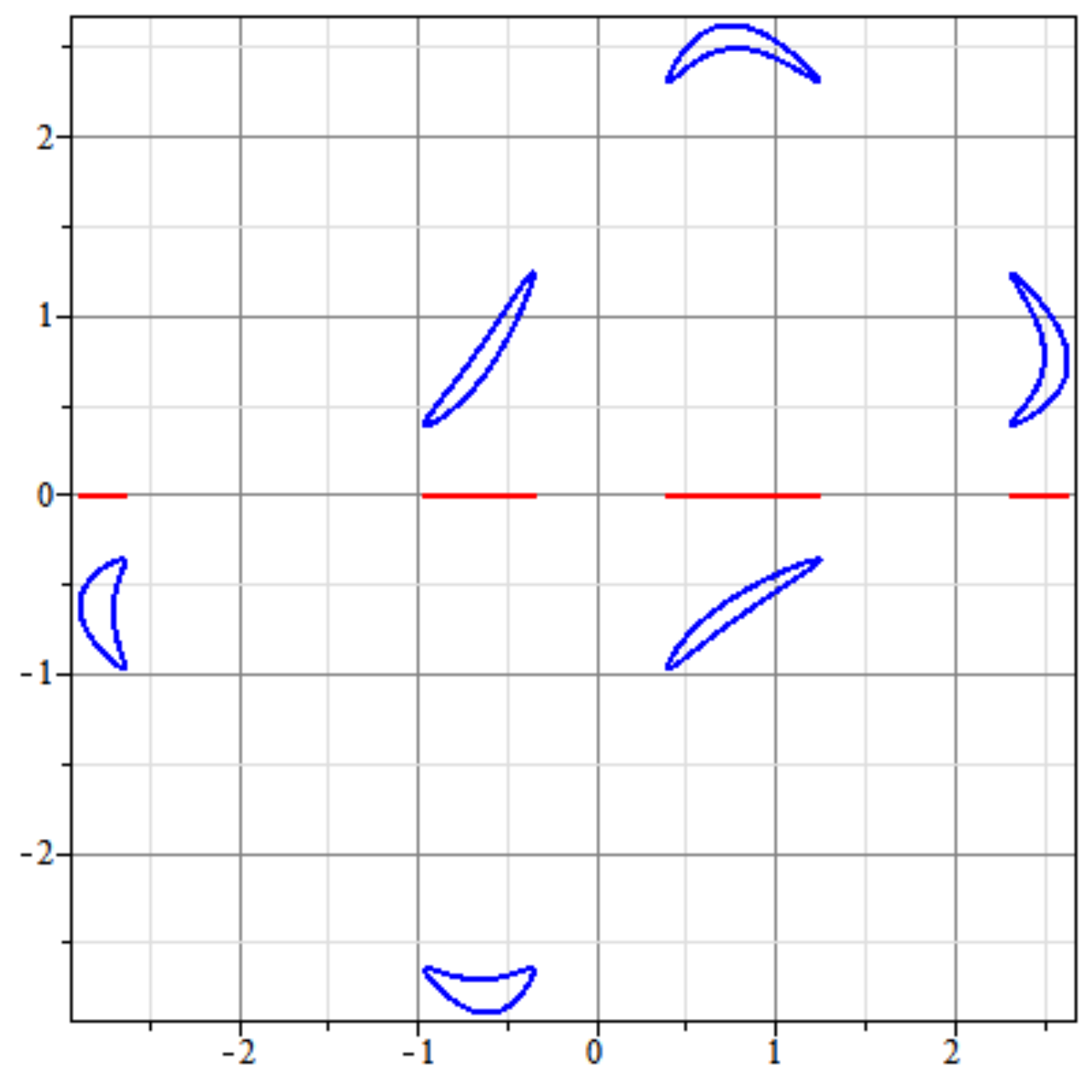}\hspace{0.5cm}
\includegraphics[scale=0.45]{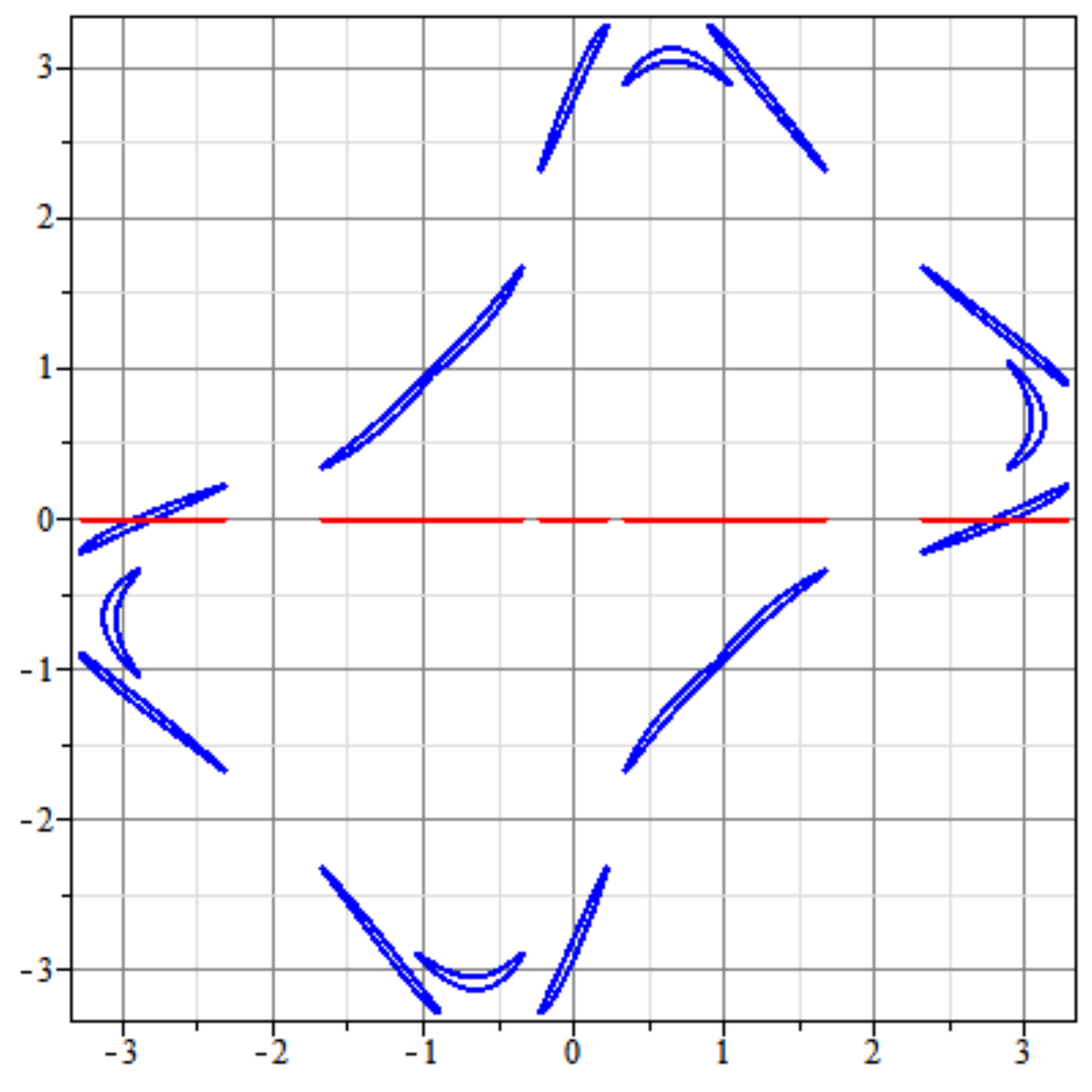}\\
\hspace{0.4cm} (a) \hspace{6.2cm} (b) \hspace{3.4cm}
\end{center}
\begin{center}
Figure 6: Orbits of  $F_{9/2,4}$ and their images under $F_4$
corresponding to the initial conditions $(2.5,0.5)$ and $(3,0.5)$
respectively, and their corresponding projections on the abscissa
axis.
\end{center}

\subsection{Local study of $F_{b,a}$ near the origin}

The origin is an elliptic fixed point of the area preserving map $F_{b,a}$ when $ab<4.$ In fact the characteristic polynomial of $DF_{b,a}(0,0)$ is
\[
P_{a,b}(\lambda)=\lambda^2+(2-ba)\lambda+1
\]
and $\mathrm{Spec}(DF_{b,a}(0,0))=\left\{\lambda^+,\lambda^-
\right\}$, where $\lambda^\pm:=-1+\frac{ab}{2}\pm \frac{i}{2}\,\sqrt
{ab(4-ab)}$.

 Computing the resultants
\[
\mbox{Res}(P_{a,b}(\lambda),\lambda^k-1;\lambda),\quad\mbox{for}\quad k=1,2,3
\]
we know that for $ab\ne1$ the eigenvalues $\lambda^\pm$  are under  the non-resonance conditions  of Moser's Stability Theorem, see \cite{AM,RCZ}. That is, they are not $k^{\mbox{th}}$-roots of the unity for $k=1,2,3.$ We remark that Example \ref{exa} seems to show that the situation $ab=1$ is not special at all.

We have obtained the local Birkhoff normal form of $F_{b,a}$ at the origin using \cite{DOP,OP}. It is given by $z\to \lambda^+
z\left(1+i\sigma z \bar{z}+O_3(z,\bar z)\right),$ where $z=x+iy\in\C$ and
$$
\sigma= {\frac {3\left( b+a \right) \sqrt {ab \left( 4-ab \right) } }{4b \left( 4-ab \right) }}\neq 0.$$

So, from Moser's Theorem  we know that for $ab\ne1$, in every
neighborhood of the origin, there exist infinitely  invariant closed
curves surrounding it. Although the theorem does not allow to make
the distinction between the integrable and the non-integrable cases,
by using the Poincar\'{e}-Birkhoff Theorem, see \cite{RCZ}, we can
ensure the existence of periodic orbits having all periods bigger
that a given one, between each two of these invariant closed curves.
Therefore we have proved that for the corresponding values of $a$
and $b$ the recurrence (\ref{Fab}) generates periodic sequences of
any even period, bigger that some constant, which depends on  $a$
and~$b$. These sequences correspond to the ones described in item
(i) of the introduction.

\end{document}